\documentclass[11pt]{amsart}

\usepackage[a4paper,margin=30mm]{geometry}
\usepackage{mathtools,amssymb,amsthm}
\usepackage{enumitem}
\usepackage{microtype}
\usepackage[colorlinks=true,linkcolor=blue,citecolor=blue,urlcolor=blue]{hyperref}
\usepackage[nameinlink,capitalise]{cleveref}

\usepackage{color}

\newtheorem{theorem}{Theorem}[section]
\newtheorem{proposition}[theorem]{Proposition}
\newtheorem{lemma}[theorem]{Lemma}
\newtheorem{corollary}[theorem]{Corollary}

\newtheorem{remark}[theorem]{Remark}

\AddToHook{env/proposition/begin}{\crefalias{theorem}{proposition}}
\AddToHook{env/lemma/begin}{\crefalias{theorem}{lemma}}
\AddToHook{env/corollary/begin}{\crefalias{theorem}{corollary}}
\AddToHook{env/definition/begin}{\crefalias{theorem}{definition}}
\AddToHook{env/remark/begin}{\crefalias{theorem}{remark}}

\numberwithin{equation}{section}

\newcommand{\ii}{\sqrt{-1}}
\newcommand{\Id}{\operatorname{Id}}
\newcommand{\rk}{\operatorname{rank}}
\newcommand{\tr}{\operatorname{tr}}

\newcommand{\Gr}{\operatorname{Gr}}
\newcommand{\HN}{\mathrm{HN}}
\newcommand{\HNS}{\mathrm{HNS}}

\newcommand{\End}{\operatorname{End}}
\newcommand{\codim}{\operatorname{codim}}
\newcommand{\supp}{\operatorname{Supp}}
\newcommand{\ddbar}{\partial\bar\partial}
\newcommand{\norm}[1]{\lVert #1\rVert}

\title[The sharp Bogomolov inequality]
{The Sharp Generalized Bogomolov Inequality}

\author[Chuanjing Zhang]{Chuanjing Zhang}
\address{School of Mathematics and Statistics, Ningbo University, Ningbo, 315211, P.R. China} \email{zhangchuanjing@nbu.edu.cn}
\author{Xi Zhang}
\address{School of Mathematics and Statistics, Nanjing University of Science and Technology, Nanjing 210094, P. R. China}
\email{mathzx@njust.edu.cn}


\begin{document}

\begin{abstract}
Let $(X,\omega)$ be a compact Gauduchon and astheno-K\"ahler manifold. Using $L^p$-approximate critical Hermitian structures, we establish a sharp generalized Bogomolov inequality for holomorphic vector bundles on $X$, without any semistability assumption. We further characterize the equality case by studying the Hermitian-Yang-Mills flow. If equality holds, every Harder-Narasimhan-Seshadri filtration is by holomorphic subbundles, and its stable quotients admit projectively flat Hermitian metrics. If, in addition, the bundle is not $\omega$-semistable, then $X$ is projective.
\end{abstract}

\thanks{The  authors are partially supported by NSF in China No.12431004, 12622103, 12371062 and 12141104. The research was supported by the Fundamental Research Funds for the Central Universities No.30926WZCT202601009 and the Stability Support Program for Basic Research No.WZJC202601001.}
\maketitle

\section{Introduction}

Let $(X,\omega)$ be a compact Hermitian manifold of complex dimension $n$. The metric $\omega$ is called Gauduchon if $\ddbar\omega^{n-1}=0$, and  astheno-K\"ahler if $\ddbar\omega^{n-2}=0$. We use the normalizations
\begin{equation}
 dV_\omega=\frac{\omega^n}{n!},
 \qquad
 V_\omega=\int_X dV_\omega.
\end{equation}
For a torsion-free coherent sheaf $\mathcal F$, let $c_j^{\mathrm{BC}}(\mathcal F)\in H_{\mathrm{BC}}^{j,j}(X,\mathbb R)$, $j=1,2$, denote its Chern classes in real Bott-Chern cohomology (\cite{BSW2023}).
If $\omega $ is a Gauduchon metric, define
\begin{equation}
 \deg_\omega(\mathcal F)
 =\int_Xc_1^{\mathrm{BC}}(\mathcal F)\wedge\frac{\omega^{n-1}}{(n-1)!},
 \qquad
 \mu_\omega(\mathcal F)=\frac{\deg_\omega(\mathcal F)}{\rk\mathcal F}.
\end{equation}
Since $\ddbar\omega^{n-1}=0$, the degree is independent of the Hermitian metric used to represent $c_1^{\mathrm{BC}}(\mathcal F)$.

Let $E$ be a holomorphic vector bundle of rank $r$, and let
\begin{equation}\label{HN-filtration}
 0=\mathcal{E}_0\subset \mathcal{E}_1\subset\cdots\subset \mathcal{E}_l=E
\end{equation}
be its Harder-Narasimhan filtration by saturated subsheaves with respect to the Gauduchon metric $\omega$ (\cite{Bru2001}).  Write
\begin{equation}
 \mathcal{Q}_\alpha=\mathcal{E}_\alpha/\mathcal{E}_{\alpha-1},\qquad
 r_\alpha=\rk \mathcal{Q}_\alpha,\qquad
 \mu_\alpha=\mu_\omega(\mathcal{Q}_\alpha),
\end{equation}
so that $\mu_1>\cdots>\mu_l$. For convenience, set
\begin{equation}\label{HN variance}
 \mathcal{I}_\omega(E):=\sum_{\alpha=1}^{l}
 r_\alpha\bigl(\mu_\alpha-\mu_\omega(E)\bigr)^2.
\end{equation}
This number is determined by the Harder-Narasimhan type of $E$ with respect to $\omega$, and it vanishes if and only if $E$ is $\omega$-semistable. Refine \eqref{HN-filtration} by choosing a Seshadri filtration of each semistable quotient $\mathcal{Q}_\alpha$ by saturated subsheaves. Its stable torsion-free quotients will be denoted by $\mathcal{Q}_{\alpha,\beta}$, and
\begin{equation}\label{HNS graded sheaf}
 \Gr^{\HNS}_\omega(E)=
 \bigoplus_{\alpha,\beta}\mathcal{Q}_{\alpha,\beta}
\end{equation}
is the associated Harder-Narasimhan-Seshadri graded sheaf.

When $\omega $ is astheno-K\"ahler, for any torsion-free coherent sheaf $\mathcal F$ of positive rank $s$,
we set
\begin{equation}\label{discriminant}
 \Delta_\omega(\mathcal F)
 =\int_X\left(2c_2^{\mathrm{BC}}(\mathcal F)
       -\frac{s-1}{s}c_1^{\mathrm{BC}}(\mathcal F)^2\right)
 \wedge\frac{\omega^{n-2}}{(n-2)!}.
\end{equation}
The integral is computed using smooth $d$-closed representatives of the  Bott-Chern classes and the astheno-K\"ahler condition makes it independent of these choices. If $\mathcal F$ is locally free, these classes are represented
by the Chern forms of any smooth Hermitian metric on $\mathcal F$.

The classical Bogomolov inequality (\cite{Bog1979,Lan2015,LZ2015,LZZ2017,Sim1988}) is one of the fundamental links between the semistability of holomorphic vector bundles and the nonnegativity of their Chern numbers. Nie and the second author (\cite{NZ2018}) proved  that the $\omega$-semistable holomorphic  bundle $E$ over the compact Gauduchon manifold $(X, \omega)$ admits an approximate Hermitian-Einstein structure. If $\omega$ is also astheno-K\"ahler, the Chern-Weil formula gives the Bogomolov inequality
\begin{equation}
\Delta_\omega(E) \geq 0.
\end{equation}
Without semistability, $\Delta_\omega(E)$ may be negative. In \cite{LZZ2025}, the authors proved the existence of $L^{p}$-approximate critical Hermitian structures by Uhlenbeck-Yau's continuity method (\cite{UY1986}). Using these structures, we prove the following sharp generalized Bogomolov inequality.
\begin{theorem}\label{thm1}
Let $(X,\omega)$ be a compact Gauduchon and astheno-K\"ahler manifold of complex dimension $n$,
and let $E$ be a holomorphic vector bundle. Then
\begin{equation}\label{sharp-Bogomolov}
 \Delta_\omega(E)
 \geq
 -\frac{n-1}{nV_\omega} \mathcal{I}_\omega(E).
\end{equation}
\end{theorem}

\begin{remark}\label{rem01}
The coefficient $(n-1)/n$ is optimal: for the Fubini-Study metric, direct sums of Hermitian-Einstein line bundles $\mathcal O_{\mathbb P^n}(k_j)$, with at least two distinct $k_j$, attain equality in \eqref{sharp-Bogomolov}.
\end{remark}

Furthermore, we  characterize the equality case.

\begin{theorem}\label{main thm}
Let $(X,\omega)$ be a compact Gauduchon and astheno-K\"ahler manifold of complex dimension $n$, and let $E$ be a holomorphic vector bundle on $X$. Let $\mathcal{Q}_{\alpha,\beta}$ denote the stable torsion-free quotients of a Harder-Narasimhan-Seshadri filtration of $E$ as in \eqref{HNS graded sheaf}. If equality holds in \eqref{sharp-Bogomolov}, then the following statements hold.
\begin{enumerate}[label=\textup{(\roman*)}]
 \item\label{(i)} The reflexive graded sheaf
 \begin{equation}
    \bigl(\Gr^{\HNS}_\omega(E)\bigr)^{**}
    =\bigoplus_{\alpha,\beta}\mathcal{Q}_{\alpha,\beta}^{**}
 \end{equation}
 is locally free.  Every stable summand $\mathcal{Q}_{\alpha,\beta}^{**}$ admits a projectively flat Hermitian metric.

 \item\label{(ii)} More precisely, there is a Hermitian metric $H_\infty$ on this graded bundle for which the summands are orthogonal and, writing $F_{\alpha,\beta}$ for their Chern curvatures,
 \begin{equation}\label{curvature}
  \ii F_{\alpha,\beta}
  =\left(\gamma
   +\frac{2\pi}{nV_\omega}
       (\mu_\alpha-\mu_\omega(E))
       \omega\right)\otimes\Id_{\mathcal Q_{\alpha,\beta}^{**}},
 \end{equation}
 where $\gamma=\frac{\ii}{r}\tr F_{H_\infty}$ is a real closed $(1,1)$-form.

\item\label{(iii)} The Harder-Narasimhan-Seshadri filtration of $E$ is a filtration by holomorphic subbundles. In particular, all its stable quotients $\mathcal{Q}_{\alpha,\beta}$ are locally free.

 \item\label{(iv)} Every Harder-Narasimhan quotient satisfies
 \begin{equation}\label{HNquotient-discriminant}
   \Delta_\omega(\mathcal Q_\alpha)=0.
 \end{equation}

 \item\label{(v)} If $E$ is not $\omega$-semistable, then $d\omega=0$. Moreover, $[\omega]$ is a positive multiple of a rational K\"ahler class, and hence $X$ is projective.
\end{enumerate}
Conversely, \ref{(i)}--\ref{(iii)} imply equality in \eqref{sharp-Bogomolov}.
\end{theorem}

The last assertion \ref{(v)} expresses a rigidity phenomenon specific to equality in the sharp inequality.  It also gives the following description of the case where the given Hermitian metric $\omega$ is not K\"ahler.

\begin{corollary}\label{nonkahler}
Assume the hypotheses of \cref{thm1} and $d\omega\neq0$.  Equality holds in \eqref{sharp-Bogomolov} if and only if $E$ is $\omega$-semistable and $\Delta_\omega(E)=0$. In that case there is a filtration by holomorphic subbundles
\begin{equation}
 0=\mathcal{F}_0\subset \mathcal{F}_1\subset\cdots\subset \mathcal{F}_m=E,
\end{equation}
whose quotients are stable and projectively flat, and satisfy
\begin{equation}\label{coro}
  \Delta_\omega(\mathcal{F}_j/\mathcal{F}_{j-1})=0, \qquad \frac{c_1^{\mathrm{BC}}(\mathcal{F}_j/\mathcal{F}_{j-1})}{\rk(\mathcal{F}_j/\mathcal{F}_{j-1})}
 =\frac{c_1^{\mathrm{BC}}(E)}{r}.
\end{equation}
\end{corollary}

This paper is organized as follows. In Section 2, we give a proof of \cref{thm1}. In Section 3, we prove \cref{main thm} and \cref{nonkahler}.

\subsection*{Acknowledgments}
The authors thank Professor Zhenlei Zhang for his insightful suggestions and helpful discussions.

\section{The sharp generalized Bogomolov inequality}

Let $H$ be a Hermitian metric on $E$, let $D_H$ be its Chern connection, and let $F_H$ be the curvature.  Set
\begin{equation}
 \Psi_H=\ii\Lambda_\omega F_H,
 \qquad
 F_H^\perp=F_H-\frac1r\tr(F_H)\otimes\Id_E,
 \qquad
 \Psi_H^\perp=\Psi_H-\frac1r\tr(\Psi_H)\Id_E.
\end{equation}
For an $\End(E)$-valued $(1,1)$-form $\phi$, its primitive decomposition is
\begin{equation}\label{primitive-decomposition}
 \phi=\phi_0+\frac1n(\Lambda_\omega \phi)\omega,
 \qquad \Lambda_\omega \phi_0=0.
\end{equation}
Since $|\omega|_\omega^2=n$,
\begin{equation}\label{primitive-norm}
 |\phi|^2=|\phi_0|^2+\frac1n|\Lambda_\omega \phi|^2.
\end{equation}

If $\omega$ is astheno-K\"ahler, the Chern-Weil formula in the normalization \eqref{discriminant} (see \cite{Kob1987,Lub1982}) is
\begin{equation}\label{CW-formula}
 4\pi^2\Delta_\omega(E)
 =\int_X\left(|F_H^\perp|^2
       -|\Lambda_\omega F_H^\perp|^2\right)dV_\omega.
\end{equation}
Together with \eqref{primitive-norm}, this becomes
\begin{equation}\label{CW-identity}
 4\pi^2\Delta_\omega(E)
 =\norm{(F_H^\perp)_0}_{L^2(H,\omega)}^2
 -\frac{n-1}{n}\norm{\Psi_H^\perp}_{L^2(H,\omega)}^2.
\end{equation}

In \cite{LZZ2025}, the authors study the asymptotic behavior of solution to the following perturbed Hermitian-Einstein equation on $(E, \bar{\partial }_{E})$, introduced by Uhlenbeck and Yau (\cite{UY1986}):
\begin{equation} \label{eq}
\sqrt{-1}\Lambda_{\omega } F_{H}-\lambda \cdot \textmd{Id}_E+\varepsilon \log (K^{-1}H)=0,
\end{equation}
where $K$ is any fixed background metric and $\lambda=2\pi \mu_\omega(E)/V_\omega$. For each $\mathcal{E}_{\alpha }$,  we have the associated orthogonal projection $\pi_{\alpha }^{K}:E\rightarrow E$ onto $\mathcal{E}_{\alpha }$ with respect to the metric $K$. Each $\pi_{\alpha }^{K}$ belongs to $L_{1}^{2}(X, \End(E))$. Define
\begin{equation}
\Phi_{\omega}^{\HN} (E, K)=\sum_{\alpha=1}^{l}\mu_{\omega } (\mathcal{Q}_{\alpha })(\pi_{\alpha }^{K}-\pi_{\alpha-1}^{K}).
\end{equation}
This is an $L_{1}^{2}$-bounded  Hermitian endomorphism, called the Harder-Narasimhan projection of  $(E, \bar{\partial}_{E})$ with respect to $K$. We have the following theorem.

\begin{theorem}[{\cite[Theorem 1.3]{LZZ2025}}] \label{theorem1}
Let $(X, \omega )$ be a compact Gauduchon manifold of complex dimension $n$,  $(E,\bar{\partial}_E)$ be a holomorphic vector bundle over $X$, $K$ be a fixed Hermitian metric on $E$ and $H_{\varepsilon}$ be a solution of perturbed equation (\ref{eq}). Then there exists a sequence $\varepsilon_{i}\rightarrow 0$ such that
 \begin{equation}\label{eqthm1}
 \lim_{i\rightarrow \infty }\left\| \sqrt{-1}\Lambda_{\omega }F_{H_{\varepsilon_{i}}}-\frac{2\pi }{V_{\omega }}\Phi_{\omega }^{HN}(E, K)\right\|_{L^{p}(K)}=0
 \end{equation}
 for any $0<p<+\infty $.
\end{theorem}

\begin{proof}[Proof of \cref{thm1}]

After a conformal transformation of $K$, we can assume that $\sqrt{-1}\Lambda_{\omega}\tr F_K= r\lambda$. Taking the trace of  (\ref{eq}) and applying the maximum principle, we obtain $\det(K^{-1}H_{\varepsilon})=1$. In particular, the sequence $H_j:=H_{\varepsilon_j}$ in Theorem \ref{theorem1} satisfies
\begin{equation}\label{trace-normalization}
\ii\Lambda_\omega\tr F_{H_j}
=\frac{2\pi r}{V_\omega}\mu_\omega(E).
\end{equation}
Consequently,
\begin{equation}\label{tracefree-convergence}
\lim_{j\to\infty}
\left\|
\Psi_{H_j }^{\perp}
-\frac{2\pi}{V_\omega}
\left(
\Phi^{HN}_\omega(E, K)
-\mu_\omega(E)\Id_E
\right)
\right\|_{L^2(K)}
=0.
\end{equation}
Observe that \eqref{eq} gives $\Psi_{H_j}=\lambda\Id_E-\varepsilon_j\log(K^{-1}H_j)$. Thus $\Psi_{H_j}$ is self-adjoint with respect to both $K$ and $H_j$, as is its trace-free part.  Consequently,
\begin{equation}\label{meancurvature-limit}
\lim_{j\to\infty}
\int_X |\Psi_{H_j}^{\perp}|_{H_{j}}^2dV_\omega
=\lim_{j\to\infty}
\int_X |\Psi_{H_j}^{\perp}|_{K}^2dV_\omega
=\frac{4\pi^2}{V_\omega}\mathcal{I}_\omega(E).
\end{equation}
This together with (\ref{CW-identity}) implies the sharp generalized Bogomolov inequality (\ref{sharp-Bogomolov}).
\end{proof}

\section{The equality case and projective flatness}

Because $\omega$ is Gauduchon, we can take a conformal change of the initial metric $H_0$ so that
\begin{equation}\label{trace0-normalization}
 \ii\Lambda_\omega\tr F_{H_0}
 =\frac{2\pi r}{V_\omega}\mu_\omega(E).
\end{equation}
Consider the following Hermitian-Yang-Mills flow
\begin{equation}\label{HYM-flow}
\left\{\begin{split}
H^{-1}(t)\frac{\partial H(t)}{\partial t}
 &=-2\Big(\ii\Lambda_\omega F_{H(t)}
       -\frac{2\pi}{V_\omega}\mu_\omega(E)\Id_E\Big),\\
 H(0)&=H_0.
 \end{split}\right.
\end{equation}
The solution exists for all $t\geq0$ (\cite{Zhang2005}).  The normalization \eqref{trace0-normalization} and the scalar heat equation for the trace imply
\begin{equation}\label{trace-fixed}
 \det H(t)=\det H_0,
 \qquad
 \tr F_{H(t)}=\tr F_{H_0}.
\end{equation}
Moreover, the maximum principle gives
\begin{equation}\label{meancurvature-bdd}
 \sup_{X\times[0,\infty)}|\Psi_{H(t)}|_{H(t)}<\infty.
\end{equation}

Let $r=\rk E$. For any Hermitian metric $H$ on $E$, we denote the $r$ eigenvalues of the mean curvature $\Psi_H$  by $\lambda_1(H,\omega)$, $\lambda_2(H,\omega)$, $\cdots$, $\lambda_r(H,\omega)$, listed in nonincreasing order. Set
\begin{equation}
\vec\lambda(H,\omega)=(\lambda_1(H,\omega),\lambda_2(H,\omega),\cdots,\lambda_r(H,\omega)).
\end{equation}
Repeating each HN slope $\mu_\alpha$ exactly $r_\alpha$ times, one can construct the nonincreasing $r$-tuple
\begin{equation}\vec{\mu}_\omega(E )=(\hat{\mu}_{1, \omega} , \cdots , \hat{\mu}_{r, \omega}),
\end{equation}
called the Harder-Narasimhan type of $E$.

In \cite{CLZZ2026}, the authors derive the following limiting property of the eigenvalue functions of mean curvature along the Hermitian-Yang-Mills flow.

\begin{theorem}[{\cite[Theorem 1.4]{CLZZ2026}}] \label{thm3}
Let $(X, \omega )$ be a compact Gauduchon manifold, $E$ a holomorphic vector bundle of rank $r$ over $X$. Assume $H(t)$ is a smooth solution of the Hermitian-Yang-Mills flow (\ref{HYM-flow}). Then
\begin{equation}\label{eigenvalue-convergence}
\lim_{t\rightarrow \infty }\left\|\vec\lambda(H(t), \omega )-\frac{2\pi }{V_{\omega}}\vec\mu_{\omega }(E)\right\|_{L^{2}}=0.
\end{equation}
\end{theorem}

\medskip

On the fixed Hermitian bundle $(E,H_0)$,  we write $\mathcal{A}_{H_0}$ for the space of connections compatible with $H_{0}$, and $\mathcal{A}_{H_0}^{1,1}$ for the space of $H_0$-unitary integrable connections. The Hermitian-Yang-Mills flow (\ref{HYM-flow}) with initial metric $H_{0}$ is gauge equivalent to the following modified Yang-Mills flow with initial connection ${A}_{0}=D_{H_{0}}\in \mathcal{A}_{H_0}^{1,1}$:
\begin{equation}
 \left\{\begin{split}\label{modified flow}
\frac{\partial A(t)}{\partial t} &=-\sqrt{-1}(\partial_{A(t)}-\bar{\partial}_{A(t)})\Lambda_{\omega}F_{A(t)},\\
A(0)&=A_0.
 \end{split}\right.
 \end{equation}
When $\omega$ is a K\"ahler metric, the above flow (\ref{modified flow}) is just the Yang-Mills flow which was introduced by  Atiyah and Bott in \cite{AB1983}.  Nie and the second author (\cite{NZ2020}) studied the limiting behavior of the  modified Yang-Mills flow (\ref{modified flow}) under the assumption that the Hermitian metric $\omega $  is Gauduchon and astheno-K\"ahler. They proved that, for every sequence $t_k\to \infty$, there is a subsequence $t_{k_j}\to \infty$ and a closed subset $\Sigma_{an}\subset X$ with Hausdorff codimension at least $4$ such that  $A(t_{k_j})|_{X\setminus \Sigma_{an}}$ converges (locally modulo unitary gauge transformations) to a connection $A_{\infty}$ satisfying
\begin{equation}\label{eq17}
 D_{A_{\infty}}\Lambda_{\omega } F_{A_{\infty}}=0.
\end{equation}
The Atiyah-Bott-Bando-Siu question was first raised by Atiyah and Bott (\cite{AB1983}) for Riemann surfaces. For reflexive sheaves on compact K\"ahler manifolds, Bando and Siu (\cite{BS1994}) asked whether the limiting reflexive sheaf is isomorphic to the double dual of the Harder-Narasimhan-Seshadri graded sheaf of the initial sheaf. This question is solved for K\"ahler manifolds (\cite{Das1992, DW2004, Jac2016, Sib2015, LZZ2018}).  In \cite{CLZZ2026}, the authors further investigate the structure of the Uhlenbeck limit $A_{\infty}$, and give an affirmative answer to the  Atiyah-Bott-Bando-Siu question in the non-K\"ahler case.

 \begin{theorem}[{\cite[Theorem 1.9]{CLZZ2026}}]\label{ABBS thm}
 Let $(E, H_{0})$ be a Hermitian vector bundle on a compact Gauduchon and astheno-K\"ahler manifold $(X,\omega)$, and $A(t)$ be a global smooth solution of the modified Yang-Mills flow \eqref{modified flow} with smooth initial data $A_{0}\in \mathcal A_{H_0}^{1,1}$. Then the extended reflexive sheaf of any Uhlenbeck limit $(E_\infty, \bar{\partial}_{A_\infty})$ is isomorphic to the double dual of the graded sheaf associated to the Harder-Narasimhan-Seshadri filtration of the holomorphic vector bundle $(E, \bar{\partial}_{A_0})$ with respect to $\omega$, that is,
 \begin{equation}\label{HNS}
 (E_{\infty}, \bar{\partial}_{A_\infty})\cong \Gr_\omega^{\HNS}(E, \bar{\partial}_{A_0})^{\ast\ast}.
 \end{equation}
  \end{theorem}

Define
\begin{equation}\label{defect}
 \mathfrak D_\omega(E)
 =\Delta_\omega(E)
  +\frac{n-1}{nV_\omega} \mathcal{I}_\omega(E).
\end{equation}

\begin{proposition}\label{flow-defect}
Along the Hermitian-Yang-Mills flow \eqref{HYM-flow}, we have
\begin{equation}\label{mean-energy-limit}
 \lim_{t\to\infty}\norm{\Psi_{H(t)}^\perp}_{L^2(H(t),\omega)}^2
 =\frac{4\pi^2}{V_\omega}\mathcal{I}_\omega(E),
\end{equation}
and
\begin{equation}\label{primitive-energy-limit}
 \lim_{t\to\infty}
 \norm{(F_{H(t)}^\perp)_0}_{L^2(H(t),\omega)}^2
 =4\pi^2\mathfrak D_\omega(E).
\end{equation}
Consequently, \eqref{sharp-Bogomolov} holds, and equality holds if and only if
\begin{equation}\label{primitive-decay}
 \norm{(F_{H(t)}^\perp)_0}_{L^2(H(t),\omega)}\longrightarrow0.
\end{equation}
\end{proposition}

\begin{proof}
At each point, there is
\begin{equation}
 |\Psi_{H(t)}^\perp|^2
 =\sum_{a=1}^r\bigl(\lambda_a(H(t), \omega)-\overline\lambda (H(t), \omega )\bigr)^2,
\end{equation}
where
\begin{equation}
\overline\lambda (H(t), \omega )=\frac1r\sum_{a=1}^r\lambda_a(H(t), \omega)=\frac1r\sum_{a=1}^r\lambda_a(H(0), \omega)=\frac{2\pi}{V_\omega}\mu_\omega(E).
\end{equation}
The HN slopes satisfy
\begin{equation}
 \frac1r\sum_{a=1}^r\widehat\mu_{a, \omega}=\mu_\omega(E).
\end{equation}
Combining \eqref{eigenvalue-convergence} and the uniform bound \eqref{meancurvature-bdd}, we derive
\begin{equation}
 \begin{split}
 \lim_{t\to\infty}\int_X|\Psi_{H(t)}^\perp|^2dV_\omega
 &=V_\omega\sum_{a=1}^r
   \left(\frac{2\pi}{V_\omega}\right)^2
   \bigl(\widehat\mu_{a, \omega}-\mu_\omega(E)\bigr)^2\\
 &=\frac{4\pi^2}{V_\omega}
   \sum_{\alpha=1}^{l}r_\alpha
   \bigl(\mu_\alpha-\mu_\omega(E)\bigr)^2.
   \end{split}
\end{equation}
This is \eqref{mean-energy-limit}.  Substituting it into \eqref{CW-identity} proves \eqref{primitive-energy-limit}.  Since the left-hand side of \eqref{primitive-energy-limit} is nonnegative, the sharp inequality follows, as does the equality criterion.
\end{proof}

The equality condition gives the following local energy estimate.

\begin{lemma}[Local energy estimate]\label{local-energy lem}
Assume that equality holds in \eqref{sharp-Bogomolov}.  Let $A(t)$ be the solution of the modified Yang-Mills flow (\ref{modified flow}), and set
\begin{equation}
 \delta(T)=\sup_{s\geq T}\int_X\big|\bigl(F_{A(s)}^\perp\bigr)_0\big|_{H_{0}}^2dV_\omega.
\end{equation}
Then $\delta(T)\to0$ as $T\to\infty$.  Moreover, for every sufficiently small $R>0$ and every parabolic cylinder
$P_R(x,t):=B_R(x)\times[t-R^2,t+R^2]$ contained in $X\times[T,\infty)$,
\begin{equation}\label{local-energy-estimate}
 R^{2-2n}\int_{P_R(x,t)}|F_{A(s)}|_{H_{0}}^2dV_\omega ds
 \leq C R^4+C R^{4-2n}\delta(T),
\end{equation}
where $C$ is independent of $x$, $t$, $T$ and $R$.
\end{lemma}

\begin{proof}
Since $\big|\bigl(F_{A(s)}^\perp\bigr)_0\big|_{H_{0}}^2=\big|\bigl(F_{H(s)}^\perp\bigr)_0\big|_{H(s)}^2$, the first assertion comes from \eqref{primitive-decay}.  By \eqref{primitive-decomposition} and \eqref{trace-fixed}, one has
\begin{equation}\label{curvature-decomposition}
 F_{A(t)}
 =\bigl(F_{A(t)}^\perp\bigr)_0-\frac{\ii}{n}\Psi_{A(t)}^\perp\omega
  +\frac1r\tr F_{A(0)}\otimes\Id_E.
\end{equation}
The last term is fixed and smooth, while \eqref{meancurvature-bdd} uniformly bounds the middle term.  Hence
\begin{equation}\label{pointwise-curvature-bdd}
 |F_{A(t)}|_{H_{0}}^2\leq C\big|\bigl(F_{A(t)}^\perp\bigr)_0\big|_{H_{0}}^2+C.
\end{equation}
The time interval of $P_R$ has length $2R^2$, and its spatial volume is bounded by $CR^{2n}$. Integrating
\eqref{pointwise-curvature-bdd} therefore yields
\begin{equation}
 \int_{P_R(x,t)}|F_{A(s)}|_{H_{0}}^2dV_\omega ds
 \leq CR^2\delta(T)+CR^{2n+2}.
\end{equation}
Multiplication by $R^{2-2n}$ means \eqref{local-energy-estimate}.
\end{proof}

\begin{proposition}\label{no-bubbling}
Assume that equality holds in \eqref{sharp-Bogomolov}. Then the analytic bubbling set $\Sigma_{an}$ of every Uhlenbeck sequence $A(t_j)$, $t_j\to\infty$, is empty.  After passing to a subsequence and applying unitary gauges, the convergence to $A_\infty$ is smooth on all of $X$.
\end{proposition}

\begin{proof}
Let $\varepsilon_0$ be the constant in the parabolic small energy regularity theorem (\cite[Theorem 2.10]{NZ2020}).
First choose $R>0$ sufficiently small that $R<\operatorname{inj}(X)/2$ and $CR^4<\varepsilon_0/2$.  With this
$R$ fixed, choose $T\ge 4R^2$ so large that $CR^{4-2n}\delta(T)<\varepsilon_0/2$. Then the small energy regularity theorem yields a uniform curvature bound on $X$ for all sufficiently large $t$. Consequently,  there exist $\rho_0>0$ and $C>0$ such that, for all $x\in X$, all sufficiently large $j$, and all $0< \rho< \rho_0$, we have
\begin{equation}
 \rho^{4-2n}\int_{B_\rho(x)}|F_{A(t_j)}|^2dV_\omega\leq C\rho^{4}.
\end{equation}
Hence $\Sigma_{an}=\emptyset$. Interior parabolic estimates and gauge fixing then imply smooth subsequential convergence on all of $X$. More precisely, after passing to a subsequence, there are smooth unitary bundle isomorphisms $u_j: E_\infty \to E$ for which $u_j^*A(t_j)\to A_\infty$ in $C^\infty(X)$. In particular, $E_\infty$ and $E$ are isomorphic as smooth complex vector bundles.
\end{proof}

\begin{lemma}\label{local freeness}
Let $X$ be a complex manifold and $E$ a holomorphic vector bundle. Let
\begin{equation}
 0=\mathcal{S}_0\subset \mathcal{S}_1\subset\cdots\subset \mathcal{S}_k=E
\end{equation}
be a filtration by saturated subsheaves. Set $Q_j=\mathcal{S}_j/\mathcal{S}_{j-1}$. Suppose that $Q_j^{**}$ is locally free and
\begin{equation}
\codim_X\supp(Q_j^{**}/Q_j)\geq 3
\end{equation}
for every $j$. Then $Q_j=Q_j^{**}$ for every $j$, and all $\mathcal{S}_j$, $Q_j$ and $E/\mathcal{S}_j$ are locally free. In particular, the filtration is by holomorphic subbundles.
\end{lemma}
\begin{proof}
Set $\mathcal{G}=\bigoplus_{j=1}^{k}Q_j$ and $Z=\operatorname{Sing}(\mathcal{G})$. Since $\mathcal{G}^{**}$ is locally free,
\begin{equation}
Z=\bigcup_{j=1}^{k}\operatorname{Supp}(Q_j^{**}/Q_j),
\end{equation}
then $\operatorname{codim}_X Z\geq 3$. The argument in the proof of \cite[Proposition 2.3]{SW2015} is local and applies on any complex manifold. It shows that $Z$, if nonempty, has pure codimension two. Thus we can conclude that $Z=\varnothing$. So every $Q_j$ is locally free and $Q_j=Q_j^{**}$. The local freeness of all $\mathcal{S}_j$ and $E/\mathcal{S}_j$ follows by induction from the corresponding short exact sequences.
\end{proof}

\begin{lemma}\label{chosen JH}
Let $E$ be an $\omega$-semistable holomorphic vector bundle on a compact Gauduchon manifold $(X,\omega)$. Suppose that $E$ admits a subbundle filtration with stable quotient bundles of slope $\mu_\omega(E)$. Let $0=\mathcal{S}_0\subset \mathcal{S}_1\subset\cdots\subset \mathcal{S}_k=E$ be any Jordan-H\"older filtration, by saturated subsheaves, whose stable factors have locally free reflexive hulls. Then this Jordan-H\"older filtration is also by holomorphic subbundles and all of its stable factors are locally free.
\end{lemma}
\begin{proof}
We argue by induction on $k$. Write the given subbundle filtration as $0=G_0\subset\cdots\subset G_m=E$, with stable quotient bundles $P_j=G_j/G_{j-1}$. Since $\mathcal S_1$ is saturated in $E$, it is reflexive and hence locally free by the hypothesis on $\mathcal S_1^{**}$.

Choose the smallest $j_0$ such that $\mathcal S_1\subset G_{j_0}$. The induced morphism $\mathcal S_1\to P_{j_0}$ is nonzero and therefore is an isomorphism, since both bundles are $\omega$-stable with slope $\mu_\omega(E)$. Consequently, $G_{j_0}= G_{j_0-1}\oplus \mathcal S_1$. So $\mathcal S_1$ is a holomorphic subbundle of $E$. Let $p: E\to E'=E/\mathcal{S}_1$ be the quotient map. The images $p(G_i)$ satisfy
\begin{equation}
p(G_{j_0})=p(G_{j_0-1}), \qquad p(G_i)/p(G_{i-1})\cong P_i \quad (i\ne j_0).
\end{equation}
Thus, after omitting the repeated term, they form a subbundle filtration of $E'$ with stable quotient bundles of slope $\mu_\omega(E')=\mu_\omega(E)$. In particular, $E'$ is $\omega$-semistable. The filtration $0\subset \mathcal{S}_2/\mathcal{S}_1\subset\cdots\subset \mathcal{S}_k/\mathcal{S}_1=E'$ is saturated and has the remaining stable factors $\mathcal{S}_i/\mathcal{S}_{i-1}$, whose reflexive hulls are locally free. Induction finishes the proof.
\end{proof}

We now prove the main theorem.

\begin{proof}[Proof of \cref{main thm}]
Choose a sequence $t_j\to\infty$.  By \cref{no-bubbling}, after passing to a subsequence and applying unitary gauges, $A(t_j)$ converges smoothly on $X$ to a connection $A_\infty$ on a smooth Hermitian bundle $E_\infty$.  The HNS identification \eqref{HNS} gives
\begin{equation}
 (E_\infty,\bar\partial_{A_\infty})
 \cong \bigoplus_{\alpha,\beta}\mathcal{Q}_{\alpha,\beta}^{**}.
\end{equation}
Each $\mathcal{Q}_{\alpha,\beta}^{**}$ is locally free, since it is a direct summand of a locally free sheaf. It is also $\omega$-stable with slope $\mu_\alpha$, as is $\mathcal{Q}_{\alpha,\beta}$.

Smooth convergence and \eqref{primitive-decay} imply
\begin{equation}\label{limit-primitive-zero}
 \bigl(F_{A_\infty}^\perp\bigr)_0=0.
\end{equation}
By (\ref{eq17}), the endomorphism $\Psi_{A_\infty}=\ii\Lambda_\omega F_{A_\infty}$ is parallel. Its distinct eigenvalues are $ \lambda_\alpha=2\pi\mu_\alpha/V_\omega$ with multiplicities $r_\alpha$ according to  \eqref{eigenvalue-convergence}. The parallel self-adjoint projections therefore give an orthogonal holomorphic decomposition
\begin{equation}
E_\infty= \bigoplus_{\alpha=1}^{l}B_\alpha, \qquad \sqrt{-1}\Lambda_{\omega }F_{A_\infty|_{B_\alpha}}=\lambda_\alpha\Id_{B_\alpha}.
\end{equation}
Each $B_\alpha$ is Hermitian-Einstein, hence $\omega$-semistable with slope $\mu_\alpha$. Uniqueness of the HN filtration implies
\begin{equation}
B_\alpha \cong \bigoplus_{\beta}\mathcal{Q}_{\alpha,\beta}^{**}.
\end{equation}

Let $\mathcal{S}\subset B_\alpha$ be the image of the first summand under this holomorphic isomorphism and $\pi_\mathcal{S}$ its orthogonal projection with respect to the limiting metric $H_\infty$. The subbundle degree formula gives
\begin{equation}
\begin{split}
\deg_\omega \mathcal{S}
&=\frac1{2\pi}\int_X\tr\bigl(\pi_\mathcal{S}(\ii\Lambda_\omega F_{H_\infty|_{B_\alpha}})\bigr)\,dV_\omega
  -\frac1{2\pi}\|\bar\partial_{B_\alpha}\pi_\mathcal{S}\|_{L^2}^2\\
&=\frac{\lambda_\alpha\rk(\mathcal{S})V_\omega}{2\pi}
  -\frac1{2\pi}\|\bar\partial_{B_\alpha}\pi_\mathcal{S}\|_{L^2}^2.
 \end{split}
\end{equation}
Since $\mu_\omega(\mathcal{S})=\mu_\alpha=\lambda_\alpha V_\omega/(2\pi)$, $\bar\partial_{B_\alpha}\pi_\mathcal{S}=0$. Hence $\mathcal{S}$ and $\mathcal{S}^\perp$ are holomorphic parallel subbundles. The quotient map restricts to a holomorphic isomorphism
\begin{equation}
\mathcal{S}^\perp\cong B_\alpha/\mathcal{S}
\cong\bigoplus_{\beta\geq2}\mathcal{Q}_{\alpha,\beta}^{**},
\end{equation}
and $\mathcal{S}^\perp$ remains Hermitian-Einstein with constant $\lambda_\alpha$. Repeating the argument deduces an orthogonal parallel decomposition with stable factors isomorphic to $\mathcal Q_{\alpha,\beta}^{**}$.

Set $\gamma=\frac{\ii}{r}\tr F_{A_\infty}$ and
\begin{equation}
 \overline\lambda
 =\frac{1}{r}\sum_\alpha r_\alpha\lambda_\alpha
 =\frac{2\pi}{V_\omega}\mu_\omega(E).
\end{equation}
Combining \eqref{primitive-decomposition} and \eqref{limit-primitive-zero}, we obtain on every stable summand
\begin{equation}\label{curvature-proof}
\begin{split}
\ii F_{\alpha,\beta}
&=\gamma\otimes\Id_{\mathcal Q_{\alpha,\beta}^{**}}
+\frac1n(\lambda_\alpha-\overline\lambda)
 \omega\otimes\Id_{\mathcal Q_{\alpha,\beta}^{**}}\\
&=\left(\gamma+\frac{2\pi}{nV_\omega}
\bigl(\mu_\alpha-\mu_\omega(E)\bigr)
\omega\right)\otimes\Id_{\mathcal Q_{\alpha,\beta}^{**}}.
\end{split}
\end{equation}
This is \eqref{curvature}.  Thus every stable summand is projectively flat.  The form $\gamma$ is real and closed by the Bianchi identity for the determinant connection. This proves \ref{(i)} and \ref{(ii)}.

Now we prove \ref{(v)}. The scalar $(1,1)$-form on the summand $Q_{\alpha,\beta}^{**}$ is
\begin{equation}
 \gamma_\alpha
 =\gamma+\frac{2\pi}{nV_\omega}
 \bigl(\mu_\alpha-\mu_\omega(E)\bigr)\omega.
\end{equation}
The Bianchi identity derives $d\gamma_\alpha=0$.  Since $d\gamma=0$,
\begin{equation}\label{bianchi-rigidity}
 \bigl(\mu_\alpha-\mu_\omega(E)\bigr)d\omega=0.
\end{equation}
 If $E$ is not semistable, at least one $\mu_\alpha$ differs from $\mu_\omega(E)$, so $d\omega=0$. Choose $\alpha$, $\eta$ with $\mu_\alpha>\mu_\eta$, and stable summands of the corresponding blocks.  Noting \eqref{curvature-proof}, one has
 \begin{equation}\label{rational-kahler-class}
  \frac{c_1(Q_{\alpha,\beta}^{**})}{\rk Q_{\alpha,\beta}^{**}}
  -\frac{c_1(Q_{\eta,\tau}^{**})}{\rk Q_{\eta,\tau}^{**}}
  =\frac{\mu_\alpha-\mu_\eta}{nV_\omega}[\omega]
  \quad\text{in }H^2_{\mathrm{dR}}(X,\mathbb R).
 \end{equation}
The left-hand side is a rational class and the coefficient on the right is positive.  Thus a positive multiple of $[\omega]$ is rational. From the Kodaira embedding theorem, we know $X$ is projective. This proves \ref{(v)}.

Next prove \ref{(iii)}. We distinguish two cases according to whether or not $E$ is $\omega$-semistable. First suppose that $E$ is not $\omega$-semistable. Then $\omega$ is K\"ahler due to \ref{(v)}. Set
\begin{equation}
 T_{\alpha,\beta}=Q_{\alpha,\beta}^{**}/Q_{\alpha,\beta}.
\end{equation}
The sheaf $T_{\alpha,\beta}$ is supported in codimension at least two. Let $Z_{\alpha,\beta}^{(2)}$
be its associated effective codimension-two cycle, with multiplicities defined as in the erratum to \cite{SW2015}. By additivity of the Chern character and \cite[Proposition 3.1]{SW2015}, we have
\begin{equation}
 \operatorname{ch}_2(Q_{\alpha,\beta}^{**})
 -\operatorname{ch}_2(Q_{\alpha,\beta})
=\operatorname{ch}_2(T_{\alpha,\beta})
=[Z_{\alpha,\beta}^{(2)}] \quad\text{in }H^4(X,\mathbb{Q}).
\end{equation}
\cref{no-bubbling} tells us that $E_\infty$ and $E$ are isomorphic as smooth complex vector bundles. Hence $\operatorname{ch}_2(E_\infty)=\operatorname{ch}_2(E)$.  On the other hand, from additivity of the Chern character in the HNS filtration, it follows that
\begin{equation}
 \operatorname{ch}_2(E)
 =\sum_{\alpha,\beta}\operatorname{ch}_2(Q_{\alpha,\beta}),
 \qquad
 \operatorname{ch}_2(E_\infty)
 =\sum_{\alpha,\beta}\operatorname{ch}_2(Q_{\alpha,\beta}^{**}).
\end{equation}
Then $0=\operatorname{ch}_2(E_\infty)-\operatorname{ch}_2(E)=\sum_{\alpha, \beta}[Z_{\alpha,\beta}^{(2)}]$. Now $\omega^{n-2}/(n-2)!$ is closed, so
\begin{equation}
0=\sum_{\alpha,\beta}\int_{Z_{\alpha,\beta}^{(2)}}\frac{\omega^{n-2}}{(n-2)!}.
\end{equation}
Positivity of $\omega$ and effectivity force each such cycle to vanish. Therefore every $T_{\alpha,\beta}$ is supported in codimension at least $3$. \cref{local freeness} then yields $T_{\alpha,\beta}=0$ for all $\alpha$, $\beta$, and shows that the filtration is by holomorphic subbundles. Now suppose that $E$ is $\omega$-semistable. Then $\mathcal{I}_\omega(E)=0$, and equality in \eqref{sharp-Bogomolov} gives $\Delta_\omega(E)=0$. According to \cite[Theorem 1.7]{PZZ2023}, it admits a filtration by holomorphic subbundles with stable quotient bundles of slope $\mu_\omega(E)$. Applying \cref{chosen JH} completes the semistable case and hence the proof of \ref{(iii)}.

It remains to prove \ref{(iv)}. For a rank $s$ bundle $\mathcal{F}$, write
\begin{equation}
\mathcal{D}(\mathcal{F}):=2c_2^{\operatorname{BC}}(\mathcal{F})-\frac{s-1}{s}c_1^{\operatorname{BC}}(\mathcal{F})^2
=\frac{c_1^{\operatorname{BC}}(\mathcal{F})^2}{s}-2\operatorname{ch}_2^{\operatorname{BC}}(\mathcal{F}).
\end{equation}
For its subbundle filtration with quotient bundles $\mathcal{G}_j$ of ranks $s_j$, set $\xi_j=c_1^{\operatorname{BC}}(\mathcal{G}_j)/s_j$, so $s=\sum_j s_j$. Due to the additivity of the Chern character in Bott-Chern cohomology,  one can see
\begin{equation}\label{chernnumber}
\mathcal{D}(\mathcal{F})=\sum_j\mathcal{D}(\mathcal{G}_j)-\frac1s\sum_{i<j}s_i s_j(\xi_i-\xi_j)^2.
\end{equation}
From \ref{(iii)}, the Jordan-H\"older filtration of each $\mathcal{Q}_\alpha$ is by subbundles. Its stable factors are projectively flat, hence $\mathcal{D}(\mathcal{Q}_{\alpha,\beta})=0$. By \eqref{curvature}, both terms on the right of \eqref{chernnumber} vanish for $\mathcal{F}=\mathcal{Q}_\alpha$. In particular, $\Delta_\omega(\mathcal{Q}_\alpha)=0$. This proves \ref{(i)}--\ref{(v)} are necessary for equality in \eqref{sharp-Bogomolov}.

Conversely, assume \ref{(i)}--\ref{(iii)}. Set $G_\alpha=\bigoplus_{\beta}\mathcal{Q}_{\alpha,\beta}$. Endow $G_\alpha$ with the restriction of the metric in \ref{(ii)} and let $F_{G_\alpha}$ be its Chern curvature. Then
\begin{equation}
 \zeta_\alpha:=\frac{\sqrt{-1}}{2\pi r_\alpha}\tr F_{G_\alpha}=\frac{\gamma}{2\pi}+\frac{\mu_\alpha-\mu_\omega(E)}{nV_\omega}\omega.
\end{equation}
These are closed real $(1,1)$-forms. \cref{(iii)} means that the Jordan-H\"older filtration of $\mathcal{Q}_\alpha$ is by holomorphic subbundles. Bott-Chern additivity therefore yields
\begin{equation}
c_1^{\operatorname{BC}}(\mathcal{Q}_\alpha)=c_1^{\operatorname{BC}}(G_\alpha)=r_\alpha[\zeta_\alpha]_{\operatorname{BC}}, \qquad \operatorname{ch}_2^{\operatorname{BC}}(\mathcal{Q}_\alpha)=\operatorname{ch}_2^{\operatorname{BC}}(G_\alpha)=\frac{r_\alpha}{2}[\zeta_\alpha^2]_{\operatorname{BC}}.
\end{equation}
Then $\mathcal{D}(\mathcal{Q}_\alpha)=0$ and
\begin{equation}
\mathcal{D}(E)=-\frac{1}{r}\sum_{\alpha<\eta}r_\alpha r_\eta[(\zeta_\alpha-\zeta_\eta)^2]_{\operatorname{BC}}.
\end{equation}
From $\zeta_\alpha-\zeta_\eta=(\mu_\alpha-\mu_\eta)\omega/(nV_\omega)$, it follows that
\begin{equation}
\begin{split}
\Delta_\omega(E)
&=-\frac{1}{rn^2V_\omega^2}\sum_{\alpha<\eta}r_\alpha r_\eta(\mu_\alpha-\mu_\eta)^2 \int_X\frac{\omega^n}{(n-2)!}\\
&=-\frac{n-1}{nV_\omega} \mathcal{I}_\omega(E).
\end{split}
\end{equation}
\end{proof}

\begin{proof}[Proof of \cref{nonkahler}]
If equality holds and $d\omega\neq0$, \cref{(v)} of \cref{main thm} implies that $E$ is $\omega$-semistable. Hence $\mathcal{I}_\omega(E)=0$ and then $\Delta_\omega(E)=0$. Conversely, semistability and $\Delta_\omega(E)=0$ immediately yield equality in \eqref{sharp-Bogomolov}.

Theorem 1.7 in \cite{PZZ2023} says that there is a filtration by holomorphic subbundles
\begin{equation}
 0=\mathcal{F}_0\subset \mathcal{F}_1\subset\cdots\subset \mathcal{F}_m=E
\end{equation}
whose quotient bundles $\mathcal{F}_j/\mathcal{F}_{j-1}$ are $\omega$-stable and satisfy
\begin{equation}
\Delta_\omega(\mathcal{F}_j/\mathcal{F}_{j-1})=0, \qquad \frac{c_1^{\operatorname{BC}}(\mathcal{F}_j/\mathcal{F}_{j-1})}{\rk(\mathcal{F}_j/\mathcal{F}_{j-1})}
 =\frac{c_1^{\operatorname{BC}}(E)}{r}  \quad\text{in }H^{1,1}_{\mathrm{BC}}(X,\mathbb R).
\end{equation}
Each stable quotient admits a Hermitian-Einstein metric (\cite{LY1987,LT2006}). For this metric, its trace-free curvature vanishes, so $\mathcal{F}_j/\mathcal{F}_{j-1}$ is projectively flat.
\end{proof}

\end{document}